\documentclass[10pt]{amsart}
\usepackage[dvips]{graphicx}
\usepackage{amsmath,graphics}
\usepackage{amscd}
\usepackage{amsfonts,amssymb}
\theoremstyle{plain}

\newtheorem{thm}{Theorem}
\newtheorem{lem}[thm]{Lemma}
\newtheorem{lemma}[thm]{Lemma}
\newtheorem{cor}[thm]{Corollary}
\newtheorem{prop}[thm]{Proposition}

\newtheorem*{claim}{Claim}
 \theoremstyle{remark}
 \newtheorem{remark}[thm]{Remark}

\newcommand{\lto}{\longrightarrow}
\newcommand{\al}{\alpha}
\newcommand{\be}{\beta}

\newcommand{\ga}{\gamma}
\newcommand{\ep}{\varepsilon}

\renewcommand{\th}{\theta}
\newcommand{\la}{\lambda}
\newcommand{\om}{\omega}

\newcommand{\si}{\sigma}
\newcommand{\De}{\Delta}

\newcommand{\Om}{\Omega}

\newcommand{\Si}{\Sigma}

\newcommand{\ZZ}{{\mathbb Z}}
\newcommand{\RR}{{\mathbb R}}
\newcommand{\CC}{{\mathbb C}}
\newcommand{\NN}{{\mathbb N}}

\newcommand{\SL}{{\rm SL}}
\newcommand{\PSL}{{\rm PSL}}
\newcommand{\GL}{{\rm GL}}
\newcommand{\tr}{\operatorname{{\it tr}}}

\newcommand{\Aut}{\operatorname{Aut}}
\newcommand{\Hom}{\operatorname{Hom}}
\newcommand{\Tor}{\operatorname{Tor}}
\newcommand{\ad}{\operatorname{{\it ad}}}
\def\co{\colon\thinspace}
\newcommand{\sm}{{\smallsetminus}}

\begin{document}
\title[Metabelian $\SL(n,\CC)$ representations of knot groups, III]{Metabelian $\boldsymbol{\SL(n,\CC)}$ representations of \\knot groups, III:  deformations}

\author{Hans U. Boden}
\address{Department of Mathematics, McMaster University,
Hamilton, Ontario} \email{boden@mcmaster.ca}
\thanks{The first author was supported by a grant from the Natural Sciences and Engineering Research Council of Canada.}

\author{Stefan Friedl}
\address{Mathematisches Institut\\ Universit\"at zu K\"oln\\   Germany}
\email{sfriedl@gmail.com}


\date{\today}
\begin{abstract}
Given a knot $K$ with complement $N_K$ and an irreducible metabelian $\SL(n,\CC)$ representation  $\al \co \pi_1(N_K) \to \SL(n,\CC)$, we establish
the inequality $ \dim  H^1(N_K;sl(n,\CC)_{\ad \al}) \geq n-1$. In the case of equality, we prove that $\al$ must have finite image and is conjugate to an $SU(n)$ representation. In this case we show $\al$ determines a smooth point $\xi_\al$ in the $\SL(n,\CC)$ character variety, and we use a deformation argument to establish the existence of a  smooth $(n-1)$--dimensional family of characters of irreducible $\SL(n,\CC)$ representations near $\xi_\al,$ and a corresponding sub--family of characters of irreducible $SU(n)$ representations of real dimension $n-1$. Both families can be chosen so that $\xi_\al$ is the only metabelian character.

Combining this with our previous existence results, we deduce the existence of large families of irreducible $SU(n)$ and $\SL(n,\CC)$ non-metabelian representation for knots $K$ in homology 3-spheres $\Si$ with nontrivial Alexander polynomial. We then
relate the condition
on twisted cohomology to a more accessible condition
on untwisted cohomology of  a certain metabelian branched cover $\widehat{\Si}_\varphi$
of $\Si$ branched along $K$.
 \end{abstract}
\maketitle

\section{Introduction}

Suppose $K$ is an oriented knot in an integral homology 3-sphere $\Si$ with exterior
$N_K=\Si^3\sm \tau(K).$
In \cite{BF08}, we show how to construct irreducible metabelian $\SL(n,\CC)$ representations
of the knot group $\pi_1(N_K)$ for any knot $K$ with nontrivial Alexander polynomial.
This provides a constructive proof for the existence of irreducible metabelian representations in
$\SL(n,\CC)$, and in this paper we prove a stronger existence result (see Theorem \ref{thm:infinitelymanyreps}) and consider the problem of existence  of irreducible non-metabelian $\SL(n,\CC)$ representations
of $\pi_1(N_K)$.

In rank $n=2$,  a result of Thurston implies that any irreducible metabelian
representation $\al \co \pi_1(N_K) \to \SL(2,\CC)$ can be deformed within the larger space of all (conjugacy classes of) representations,
and in fact Theorem 3.2.1 of \cite{CS83} shows the existence of a family of conjugacy classes of irreducible representations
near $\al$ of dimension $ \ge 1$.
In this paper, we study the character varieties of knot groups in higher rank, with a focus on existence of irreducible metabelian $\SL(n,\CC)$ representations and their deformations.
For instance, given an irreducible metabelian  representation $\al \co \pi_1(N_K) \to \SL(n,\CC)$ satisfying a cohomological condition,
we establish the existence of an $(n-1)$-dimensional  family of conjugacy classes of irreducible non-metabelian $\SL(n,\CC)$
representations near $\al$.

In order to more precisely state our results, we introduce notation that will be used throughout the paper.

Given a finitely generated group $\pi$, let $R_n(\pi) = \Hom(\pi, \SL(n,\CC))$ be
the representation variety,
which is an affine algebraic set with a natural action of $\SL(n,\CC)$ by conjugation. The set-theoretic quotient is in general not well-behaved, (e.g. it is typically not Hausdorff), so instead we consider the natural quotient in the category of algebraic sets, which is by definition the character variety $X_n(\pi)$ (see \cite{LM85} for details on the construction of character varieties).
Given a representation $\al \co \pi \to \SL(n,\CC)$, its character is the map
$\xi_\al \co \pi \to \CC$ defined by $\ga \mapsto \tr \al(\ga)$, and
setting $t(\al) = \xi_\al$ defines the quotient map
$t \co R_n(\pi) \to X_n(\pi)$.

For a topological space $M$, let $R_n(M) = R_n(\pi_1 (M))$ and $X_n(M)=X_n(\pi_1 (M))$.
Given  $\al \co \pi_1(M) \to \SL(n,\CC)$, let $\ad \al$ be its composition with
the adjoint representation  on the Lie algebra $sl(n,\CC)$, thus  $\ad \al$ determines a $\pi_1(M)$ action on $sl(n,\CC)$. We let $H^*(M;sl(n,\CC)_{\ad \al})$ denote the
 cohomology groups of $M$ with coefficients in $sl(n,\CC)$ twisted by this action.

Given  an irreducible metabelian representation $\al \co \pi_1(N_K) \to \SL(n,\CC)$, we show that  $\dim H^1(N_K; sl(n,\CC)_{\ad \al}) \ge n-1$ (Proposition \ref{prop-Lag}) and deduce that $\dim X_j \geq n-1$ for any algebraic component $X_j \subset X_n(N_K)$  containing $\xi_\al$ (Corollary \ref{cor-dim-alg-comp}).
We then show that if $\al \co \pi_1(N_K) \to \SL(n,\CC)$ is an irreducible metabelian representation
such that  $\dim H^1(N_K; sl(n,\CC)_{\ad \al}) = n-1$, then $\al$
has finite image and is conjugate to a unitary representation.
The following result gives a local description of the character variety near $\xi_\al$ under the assumption
$\dim H^1(N_K; sl(n,\CC)_{\ad \al}) = n-1$.

\begin{thm}\label{thm1}
If $\al \co \pi_1(N_K) \to \SL(n,\CC)$ is an irreducible metabelian representation
with $\dim H^1(N_K; sl(n,\CC)_{\ad \al}) = n-1$, then $\al$ has finite image and is therefore conjugate to a unitary representation. Further, we have:
\begin{enumerate}
\item[(i)] The character $\xi_\al$ is a smooth point in $X_n(N_K)$, and there
exists a smooth complex $(n-1)$--dimensional family of characters of
irreducible $\SL(n,\CC)$ representations near $\xi_\al \in X_n(N_K)$.
\item[(ii)] As a point in $X_{SU(n)}(N_K)$, the character $\xi_\al $ is again a smooth point
and there exists a smooth real $(n-1)$--dimensional family of characters of irreducible $SU(n)$
representations near $\xi_\al \in X_{SU(n)}(N_K)$.
\end{enumerate}
Both deformation families can be chosen so that
$\xi_\al$ is the only metabelian character within them.
 \end{thm}

Deformations of  dihedral $\SL(2,\CC)$ representations were studied by Heusener and Klassen in \cite{HK97}, and metabelian representations in $\SL(n,\CC)$ are their analogues in higher rank. Theorem \ref{thm1} is established by applying
deformation arguments developed for $\SL(2,\CC)$ and $\PSL(2,\CC)$ by Heusener, Porti, and  Su\'{a}rez Peir\'{o}
in \cite{HPS01, HP05}. These techniques were extended to $\SL(n,\CC)$ in \cite{AHJ10}, where they were applied to deform reducible metabelian $\SL(3,\CC)$ representations of knot groups.  In this paper, we apply the same technique to the problem of deforming  irreducible metabelian characters.
In Subsection \ref{sec-simple}, we state the deformation results  that are needed to establish Theorem \ref{thm1}, and in Appendix \ref{app-deform}, we provide detailed arguments for these results, following the treatment given in  \cite{HPS01, HP05, AHJ10}.

Theorem \ref{thm1} applies in many cases. For instance, in rank $2$, given an irreducible
representation $\al \co \pi_1(N_K) \to \SL(2,\CC)$ such that $\{\al(\mu), \al(\la) \} \not\subset \{ \pm I \},$
Thurston proved that any algebraic component  of $X_2(N_K)$ containing $\xi_\al$ has dimension $d\geq 1$
(see Theorem 3.2.1 of \cite{CS83}). We will see that every irreducible metabelian $\SL(2,\CC)$ representation
$\al$ satisfies this condition, and thus $\xi_\al \in X_2(N_K)$ can be deformed to an irreducible non-metabelian representation. If one assumes, in addition, that $N_K$ does not contain any closed incompressible surfaces, then
it follows from \cite{CS83} that any algebraic component   $X_2(N_K)$ has dimension $ d=1.$
Knots $K$ whose complements $N_K$ satisfy this condition are called {\em small}, and we see that Theorem \ref{thm1} applies to irreducible metabelian representations
$\al\co \pi_1(N_K) \to \SL(2,\CC)$ when $K$ is a  small knot  in an integral homology 3-sphere.
 In Subsection \ref{sec-examples}, we show by example that Theorem \ref{thm1} can also be applied in higher rank.

Note that Theorem \ref{thm1} does not apply
to  knots $K$ whose Alexander polynomial $\De_K(t)$ has a root which is an $n$-th root of unity.
 Indeed, if $L_n$ is the $n$-fold cyclic branched cover of $\Si$ branched along $K$, then we have $b_1(L_n) >0$, and
 any irreducible metabelian representation  $\al \co \pi_1(N_K) \to \SL(n,\CC) $
will have $\dim H^1(N_K; sl(n,\CC)_{\ad \al}) > n-1$ (see Proposition \ref{prop-finite}).
The simplest example occurs in rank $n=6$ for the trefoil knot $K=3_1,$
though other examples can be constructed using Theorem 3.10 of \cite{BF08}.

Thus, it is useful to have an alternative criterion for applying Theorem \ref{thm1}, and our next result  provides such a criterion in terms of the untwisted cohomology of a certain metabelian branched cover of $\Si$ branched along $K$.

\begin{thm} \label{thm2}
Suppose that $n$ is such that $b_1(L_n)=0$ (equivalently,  suppose the Alexander polynomial
$\De_K(t)$ has no root which is an $n$-th root of unity).
Suppose further that  $\al \co\pi_1(N_K)\to \SL(n,\CC)$ is an irreducible metabelian representation
 and $\varphi\co \pi_1(N_K) \to \ZZ/n \ltimes H$ is a group homomorphism with $H$ finite and abelian
such that $\al$ factors through $\varphi$.
Denote by $\widetilde{N}_\varphi\to N_K$ the covering map corresponding to $\varphi$. Then the following hold:
\begin{enumerate}
\item[(i)]  $b_1(\widetilde{N}_\varphi) \geq |H|$ and
if $b_1(\widetilde{N}_\varphi)=|H|$, then $\dim H^1(N_K; sl(n,\CC)_{\ad \al})= n-1$.
\item[(ii)] The cover $\widetilde{N}_\varphi\to N_K$ extends to a  cover $\widehat{\Si}_\varphi\to \Si$ branched over $K$.
\item[(iii)] If $b_1(\widehat{\Si}_\varphi)=0$, then $\dim H^1(N_K; sl(n,\CC)_{\ad \al})= n-1$.
\end{enumerate}
\end{thm}

 \begin{remark}
Theorem \ref{thm2} is a generalization of a result proved for dihedral groups by   Boileau and Boyer, see
 \cite[Lemma~A.2]{BB07}.
 \end{remark}

\bigskip \noindent
{\em Acknowledgments.  } The authors are grateful to  Steve Boyer, J\'er\^{o}me Dubois, Michael Heusener, and Adam Sikora for many
helpful discussions. The first author is happy to acknowledge the Max Planck Institute for Mathematics for its support.

\section{Metabelian representations of knot groups}

In this section we review the construction of metabelian representations for knot groups from \cite{BF08}.
We then use a result of Silver and Williams \cite{SW02} to show
existence of irreducible metabelian $\SL(n,\CC)$ representations for all but finitely many ranks $n$ for any knot $K$
whose Alexander polynomial $\Delta_K(t)$ has a root that is not a root of unity.

\subsection{Construction of irreducible metabelian SL$\boldsymbol{(n, \CC)}$ representations}
\label{sec2-3}

Given a group $\pi$ and a finite dimensional vector space $V$ over $\CC$,
a representation $\varrho \co \pi\to \Aut(V)$
is called \emph{reducible} if there exists a proper invariant subspace $U\subset V$, otherwise
$\varrho$  is called \emph{irreducible}.
We say $\varrho$
is \emph{metabelian}
 if its restriction $\varrho |_{ \pi^{(2)}}$ is trivial, where $\pi^{(2)}$ denotes the second commutator subgroup of $\pi$. Equivalently, a metabelian representation is one that factors
   through the metabelian quotient $\pi/\pi^{(2)}.$

Given a knot $K\subset \Si^3$ in an integral homology 3-sphere, let $N_K = \Si \sm \tau(K)$ be the complement and
$\widetilde{N}_K$ be the infinite cyclic cover of $N_K$.
Thus $\pi_1(\widetilde{N}_K)=\pi_1(N_K)^{(1)}$ and
$$ H_1(N_K;\ZZ[t^{\pm 1}])=H_1(\widetilde{N}_K) \cong \pi_1(N_K)^{(1)}/\pi_1(N_K)^{(2)},$$
where we use $\pi^{(n)}$ to denote the $n$--th term of the
derived series of a group $\pi$, so $\pi^{(1)}=[\pi,\pi]$ and $\pi^{(2)} = [\pi^{(1)},\pi^{(1)}]$, and so on.
The $\ZZ[t^{\pm 1}]$--module structure is
given on the right hand side
by $t^n\cdot g:=\mu^{-n}g\mu^n$, where $\mu$ is a meridian
of $K$.

Set $\pi:=\pi_1(N_K)$ and $H = H_1(N_K;\ZZ[t^{\pm 1}])$ and consider the short exact sequence
$$ 1\to \pi^{(1)}/\pi^{(2)}\to \pi/\pi^{(2)}\to\pi/\pi^{(1)}\to 1. $$
Since $\pi/\pi^{(1)}=H_1(N_K)\cong \ZZ$, this sequence splits and we get isomorphisms
$$ \begin{array}{rcccl} \pi/\pi^{(2)}
&\cong & \pi/\pi^{(1)} \ltimes \pi^{(1)}/\pi^{(2)}
&\cong &\ZZ \ltimes \pi^{(1)}/\pi^{(2)}   \cong \ZZ \ltimes H \\
    g &\mapsto &(\mu^{\ep(g)},\mu^{-\ep(g)}g) &\mapsto &(\ep(g),\mu^{-\ep(g)}g), \end{array}  $$
where the semidirect products are taken with respect to the $\ZZ$ actions defined by
letting $n \in \ZZ$ act by conjugation by $\mu^n$ on $\pi^{(1)}/\pi^{(2)}$ and by multiplication
by $t^n$ on $H_1(N_K; \ZZ[t^{\pm1}])$. This demonstrates the following lemma.

\begin{lem}\label{lem4}
For any knot $K$, the set of
metabelian representations of $\pi_1(N_K)$ can be canonically identified with the set of
representations of $\ZZ \ltimes H$.
\end{lem}

When it is convenient, we will  blur the distinction between metabelian representations of $\pi_1(N_K)$ and representations of $ \ZZ \ltimes H$.


 Lemma \ref{lem4} applies to give a useful classification of the irreducible  $\SL(n,\CC)$ of $\pi_1(N_K)$, and before explaining that, we
point out two important and well--known facts that are used frequently:
\begin{enumerate}
\item[(i)]  $H=H_1(N_K;\ZZ[t^{\pm 1}])$ is finitely generated as a
$\ZZ[t^{\pm 1}]$--module and multiplication by $t-1$ is an isomorphism.
\item[(ii)] There is an isomorphism $H/(t^n-1)\cong H_1(L_n)$, where $L_n$ denotes the $n$--fold cyclic branched cover of $\Si^3$ branched along $K$.
\end{enumerate}

Suppose $\chi \co H\to \CC^*$ is a character factoring through $H/(t^n-1)$ and $z\in U(1)$ satisfies $z^n=(-1)^{n+1}$. Given $(j, h) \in \ZZ\ltimes H,$ we set
$$  \al_{(n,\chi,z)} (j,h) =
 \begin{pmatrix}
 0& &\dots &z \\
 z&0&\dots &0 \\
\vdots &\ddots &\ddots&\vdots \\
     0&\dots &z &0 \end{pmatrix}^j
     \begin{pmatrix} \chi(h) &0&\dots &0 \\
 0&\chi(th) &\dots &0 \\
\vdots &&\ddots &\vdots \\ 0&0&\dots &\chi(t^{n-1}h) \end{pmatrix}.
$$
It follows that $ \al_{(n,\chi,z)}$
defines an  $\SL(n,\CC)$ representation that factors over $\ZZ\ltimes H/(t^n-1)$ and whose  isomorphism type  is independent of the choice of $z$ (see
\cite[Section~3]{BF08}). We write $\al_{(n,\chi)}$ for $\al_{(n,\chi,z)}$.

Recall that a character $\chi \co H\to \CC^*$ has \emph{order $n$} if
it factors through $H/(t^n-1)$ but not through $H/(t^\ell-1)$ for any $\ell < n$.
Any character $\chi \co H\to \CC^*$ which factors through $H/(t^n-1)$ must have order $k$ for some
divisor $k$ of $n$.

 Given a character  $\chi \co H\to \CC^*$, let $t^i\chi$ be the character defined by $(t^i\chi)(h)=\chi(t^ih)$.
The next theorem gives a summary of the results \cite[Lemma~2.2]{BF08} and  \cite[Theorem~3.3]{BF08}.

\begin{thm} \label{thm-BF}
Suppose  $\chi \co H  \to \CC^*$ is a character that factors through $H/(t^n-1)$.
\begin{enumerate}
\item[(i)] $\al_{(n,\chi)}\co \ZZ\ltimes H\to \SL(n,\CC)$ is irreducible if and only if the character $\chi$ has order $n$.
 \item[(ii)] Given two characters $\chi,\chi'\co H\to \CC^*$ of order $n$, the representations
$\al_{(n,\chi)}$ and $\al_{(n,\chi')}$ are conjugate if and only if $\chi=t^k\chi'$ for some $k$.
\item[(iii)] For any irreducible representation  $\al \co\ZZ\ltimes H\to \SL(n,\CC)$
 there exists a character $\chi \co H\to \CC^*$ of order $n$ such that
$\al$ is conjugate to $\al_{(n,\chi)}$.
\end{enumerate}
\end{thm}

\begin{remark} \label{remark-regular}
Note that
$$\al_{(n,\chi)}(\mu)    = \begin{pmatrix}
  0 &  \dots & 0 &z \\
  z&0 &  \dots &0 \\
  &\ddots &\ddots & \vdots \\
    0& &z & 0
  \end{pmatrix}$$
  is conjugate to the diagonal matrix
  $$\begin{pmatrix} z&&& 0 \\
 &\om z  \\
&&\ddots  \\
     0&&&\om^{n-1} z \end{pmatrix},
$$
where $z$ satisfies $z^n = (-1)^{n+1}$ and $\om = e^{2 \pi i/n}.$ In particular, this shows under $\al_{(n,\chi)}$, the meridian is sent to a matrix with $n$ distinct eigenvalues.
\end{remark}


\subsection{Existence of irreducible metabelian SL$\boldsymbol{(n, \CC)}$ representations}
\label{sec2-3b}
In this section, we apply results of \cite{SW02} to prove a strong existence result for irreducible metabelian $\SL(n,\CC)$ representations of knot groups.

Suppose $K$ is a knot whose Alexander polynomial $\Delta_K(t)$ has a zero which is not a root of unity.
Then Kronecker's theorem implies that the Mahler measure $m$ of $\Delta_K(t)$ satisfies $m>1.$
Recall that the Mahler measure of a polynomial $f(t) \in \CC[t]$ is defined by the formula
$$m(f) =  \exp{\int_0^{2 \pi} \ln(|f(e^{i \th}|) d\th}.$$
The next proposition
was proved by Silver and Williams in \cite{SW02}, and it is an extension of earlier results of Gordon \cite[p.~365]{Gor72},
Gonz\'alez-Acu\~na--Short \cite{GAS91} and Riley \cite{Ri90}.

\begin{prop}[Theorem 2.1, \cite{SW02}] \label{SW-prop}
Let $K$ be a knot and let $m$ be the Mahler measure of $\Delta_K(t)$.
Then
$$ \lim_{n\to\infty} \frac{\ln \Tor H_1(L_n)}{n} =\ln m.$$
\end{prop}

We now explain how to apply Proposition \ref{SW-prop} to deduce a strengthened existence result for  irreducible  metabelian $\SL(n,\CC)$  representations for such knots $K$ (cf. Theorems 3.10 and 3.12 of \cite{BF08}).

\begin{thm}\label{thm:infinitelymanyreps}
Suppose $K$ is a knot such that $\Delta_K(t)$ has a zero which is not a root of unity.
Then the number of distinct conjugacy classes of
irreducible metabelian  $\SL(n,\CC)$  representations of the knot group increases exponentially
as $n \to \infty$. Consequently,  for all but finitely many ranks $n,$ there exist irreducible metabelian representations
$\al \co \pi_1(N_K) \to \SL(n,\CC).$
\end{thm}

\begin{proof}
Let $K$ be a knot. Given $n\in \NN$, let  $r_n\in \NN\cup \{\infty\}$ denote the number of distinct conjugacy classes of
 irreducible metabelian  $\SL(n,\CC)$  representations of the knot group.

\begin{claim}
Let $n\in \NN$ and write $n =p_1 p_2 \cdots p_k$ for primes $p_1,\dots,p_k$. Then
\[ r_n\geq \frac{1}{n}\left(|\Tor H_1(L_n)| - \sum_{i=1}^k |\Tor H_1(L_{n/p_i})|\right).\]
\end{claim}

We write $H = H_1(N_K; \ZZ[t,t^{-1}])$. Note that given any $k|n$ we have the following commutative diagram
\[ \begin{array}{ccc} H/(t^n-1)&\xrightarrow{\cong} &H_1(L_n) \\  \downarrow &&\downarrow \\
H/(t^k-1)&\xrightarrow{\cong} &H_1(L_k).\end{array}\]
We pick once and for all a decomposition $H_1(L_n)= F_n\oplus T_n$ where $F_n$ is a free abelian group and $T_n$ is torsion.
It follows from Theorem \ref{thm-BF} (i) and (ii) that
\[ \begin{array}{rcl} r_n&\geq & \frac{1}{n} \# \{\rho\co H_1(L_n)\to T_n\to S^1\,|\,\rho \mbox{ does not  factor through some $H_1(L_k)$}\}\\[2mm]
&=&\frac{1}{n}\# \{\rho\co T_n\to S^1\, |\, \rho \mbox{ does not  factor through $T_n\to H_1(L_n)\to H_1(L_k)$}\}. \end{array}\]
Note that the number of characters of a finite group $A$ equals $|A|$. Also note that any map
$T_n\to H_1(L_k)$  necessarily factors through $\Tor H_1(L_k)$.
The claim is now an immediate consequence of these observations.

\begin{claim}
Suppose $M>1$ and $n =p_1 p_2 \cdots p_k$ for primes $p_1,\dots,p_k$.
Then
$$\sum_{i=1}^k M^{n/p_i}\leq \frac{\ln n}{\ln 2}M^{n/2}.$$
\end{claim}

Since each prime factor $p_i \geq 2$, it follows that  $p_1 \cdots p_k=n\geq 2^k$.
Thus
$$\sum_{i=1}^k M^{n/p_i} \leq \sum_{i=1}^k M^{n/2}=k M^{n/2} \leq  \frac{\ln n}{\ln 2}M^{n/2},$$
and this completes the proof of the claim.

We can now finally turn to the proof of the theorem.
Suppose  that $\Delta_K(t)$ has a zero which is not a root of unity.
Let $m$ be the  Mahler measure of $\Delta_K(t)$ and notice that Kronecker's theorem implies $m>1$.

Suppose $0<\ep < 1/3$. By Proposition \ref{SW-prop}, there exists an $N$ such that  $n \geq N$ implies
$$ \left( m^{1-\ep}\right)^n \leq | \Tor H_1(L_n)| \leq \left(m^{1+\ep}\right)^n.$$
We  write
\[ D:= \sum_{i=1}^N |\Tor H_1(L_{i})|.\]
Now let $n\geq N$. We factor $n =p_1 p_2 \cdots p_k$ where $p_1,\dots,p_k$ are primes.
If we combine the above with the first claim we see that
\[ \begin{array}{rcl}r_n&\geq &\frac{1}{n}\left(|\Tor H_1(L_n)| - \sum_{i=1}^k |\Tor H_1(L_{n/p_i})|\right)\\[2mm]
&\geq & \frac{1}{n} \left(m^{(1-\ep)n}-\sum_{i=1}^k m^{(1+\ep)n/p_i}-D\right).\end{array} \]
Applying the second claim with $M =  m^{1+\ep}$, it follows that
\[ r_n\geq \frac{1}{n} \left(m^{(1-\ep)n}-\frac{\ln n}{\ln 2}m^{(1+\ep)n/2}-D\right).\]
This shows that $r_n$ grows exponentially for sufficiently large $n$.
\end{proof}


\section{Twisted homology and cohomology}
In this section, we introduce the  twisted homology and cohomology groups and give some computations that are used throughout the paper.

\subsection{The adjoint representation}

In this subsection, we show how, given a metabelian representation,  its   adjoint representation decomposes as a direct sum of  simple representations.

\begin{lemma}\label{lem-a}
Let $K$ be a knot, $n\in \NN$, and $\chi \co H_1(L_n)\to \CC^*$
a character.
Set $\al=\al_{(n,\chi)}$
and let $\th_1\co \pi_1(N_K)\to \GL(1,\CC)$ denote the trivial representation.
Let $\al_n\co \pi_1(N_K)\to \Aut(\CC[\ZZ/n])$ be the regular representation corresponding to the canonical projection map
$\pi_1(N_K)\to \ZZ\to \ZZ/n$, and let
$\ad \al\co \pi_1(N_K)\to \Aut(sl(n,\CC))$ denote the adjoint representation. Then we
have
 the following isomorphism of representations:
$$ \ad \al \oplus \th_1\cong \al_n \oplus \bigoplus\limits_{i=1}^{n-1} \be_{(n,\chi_i)},$$
where $\chi_i$ is the character defined by $\chi_i(v):=\chi(v)^{-1}\chi(t^{i}v)$.
Further, if $\chi$ is a character of order $n$, then $\chi_1,\dots,\chi_{n-1}$ are also characters of order $n$.
\end{lemma}

\begin{proof}
Write $\pi=\pi_1(N_K)$ as before
and let $\be\co \pi\to \Aut(gl(n,\CC))$ denote the adjoint representation of $\al$
on $gl(n,\CC)$,
so $\be(g)(A)=\al(g)A\al(g)^{-1}$ for $g\in \pi$ and $A\in gl(n,\CC)$.
Note that $gl(n,\CC)=sl(n,\CC)\oplus \CC\cdot I$. It follows immediately that
$\be=\ad \al \oplus \th_1$ splits off a trivial factor. It therefore suffices to show that
$$\be \cong \al_n \oplus \bigoplus\limits_{i=1}^{n-1} \be_{(n,\chi_i)}.$$
For $i=0,\dots,n-1,$ let $V_i$ be the set of all matrices
$(a_{jk})$ such that $a_{jk}=0$ unless $j-k \equiv i \mod{n}$.
It is not difficult to see that the action of $\pi$ on $gl(n,\CC)$ restricts to actions
on $V_0,V_1,\dots,V_{n-1}$.
 We equip $V_i$ with the ordered basis
$\{ e_{i+1,1},e_{i+2,2},\dots,e_{i+n,n}\}$, where the indices are taken modulo $n$.
The restriction of $\be$ to ${V_i}$ can then be calculated with respect to this basis and
$\alpha_{(n,
\chi,z)}(j,h)=z^j\beta_{(n,\chi)}(i,h)$
 and the $z^i$ disappears upon conjugation.

Note that   $\chi_0$ is the trivial character, and therefore
$\be_{(n,\chi_0)}=\al_n$.
\end{proof}

As a side note we record a corollary on twisted Alexander polynomials.
Recall that given a knot $K$ and a complex representation $\gamma$ we can consider the corresponding
twisted Alexander polynomial $\Delta_K^\gamma\in \CC(t)$. We refer to \cite{Wa94,FV10} for details.
We obtain the following corollary.

\begin{cor}
Let $K$ be an oriented knot , $n\in \NN$ and $\chi \co H_1(L_n)\to \CC^*$
be a character.
Suppose $\al=\al_{(n,\chi)}$. Then
\[ \Delta_K^{\ad \al}\doteq \prod\limits_{j=1}^{n-1} \Delta_K(e^{2\pi ij/n}t)\cdot \prod\limits_{j=1}^{n-1} \Delta_K^{\beta(n,\chi_i)}.\]
\end{cor}

This corollary generalizes a recent theorem of Yamaguchi
\cite{Ya11} which says  that in the case $n=2$ the polynomial $\Delta_K(-t)$ divides $\Delta_K^{\ad \al}$.

\subsection{Twisted homology and cohomology}

In this subsection, we recall the  twisted homology and cohomology groups and summarize their  basic properties.

Let $(X,Y)$ be a pair of topological spaces, $V$ a finite dimensional complex vector space and $\al \co \pi_1(X)\to \Aut(V)$ a representation.
Denote by $p\co \widetilde{X}\to X$ the universal covering and set $\widetilde{Y}:=p^{-1}(Y)$.
Using the representation, we can regard $V$ as a left $\ZZ[\pi]$--module, where $\pi=\pi_1(X)$.
The chain complex $C_*(\widetilde{X},\widetilde{Y})$ is also a left $\ZZ[\pi]$--module via deck transformations and we  form
the twisted cohomology groups
$$ H^*(X,Y;V_\al)=H_*(\Hom_{\ZZ[\pi]}(C_*(\widetilde{X},\widetilde{Y}),V)).$$
 Using the natural involution $g\mapsto g^{-1}$ on the group ring $\ZZ[\pi]$, we can view $C_*(\widetilde{X},\widetilde{Y})$ as a right $\ZZ[\pi]$--module, and we can form
the twisted homology groups
$$ H_*(X,Y;V_\al)=H_*(C_*(\widetilde{X},\widetilde{Y})\otimes_{\ZZ[\pi]}V).$$

The groups $H^0$ and $H_0$ can be computed
immediately from the fundamental group (cf. \cite[Section~VI]{HS97}):
\begin{equation} \label{eq1}
\begin{array}{rcl} H^0(X;V_\al)&=&\{ v\in V \mid \al(g)v=v\text{ for all }g\in \pi\},\\
 H_0(X;V_\al)&=&V/ \sim, \text{ where $\al(g)v \sim v$ for all  $v\in V, g\in \pi$.}
 \end{array}
 \end{equation}
If $M$ is an $n$--manifold, then Poincar\'e duality implies
$$ H_i(M;V_\al)\cong H^{n-i}(M,\partial M;V_\al) \text{ and } H_i(M,\partial M;V_\al)\cong H^{n-i}(M;V_\al).$$

The next two lemmas are both well--known and therefore stated without proof. For more details, see \cite[Lemma~2.3]{FK06}.

\begin{lem}\label{lem-b}
Suppose that  $V$ is equipped with a bilinear non--singular form, and
that $\al$ is orthogonal with respect to this form.
Then
$$ H_i(X,Y;V_\al)\cong H^{i}(X,Y;V_\al)$$
for any $i$.
The same conclusion holds in the case $V$ has a non--singular hermitian form and
$\al$ is unitary with respect to this form.
\end{lem}

Consider the map defined for $A,B \in sl(n,\CC)$ by the assignment
$$(A,B)\mapsto -\tr (AB).$$
This map defines a non--singular, symmetric, bilinear form on $sl(n,\CC)$ called the Killing form.
The next lemma says that the hypotheses of Lemma \ref{lem-b}
are satisfied for the adjoint representation.
\begin{lem}\label{lem-c}
For any $\al \co \pi\to \SL(n,\CC)$, its adjoint representation   $\ad \al \co \pi \to \Aut(sl(n,\CC))$ is orthogonal with respect to the Killing form.
\end{lem}

\subsection{Calculations}

This subsection  presents  some calculations of twisted homology and cohomology groups
that will be used in proving the main results.

\begin{lem}\label{lem-torus-V}
Let $K$ be a knot, $\chi \co H_1(L_n) \to \CC^*$  a character and $z\in U(1)$. Let $V=\CC^n$
and $\al=\al_{(n,\chi,z)}\co \pi_1(N_K) \to \Aut(V)$,
and  set $\widehat{\al}$ to be the restriction of $\al$ to $\pi_1(\partial N_K)$.
If $z^n=1$, then  the following hold:
$$ \begin{array}{cccccc}  \dim H^0(\partial N_K; V_{\widehat{\al}})& =& \dim H_0(\partial N_K; V_{\widehat{\al}})&=& 1,\\
 \dim H^1(\partial N_K; V_{\widehat{\al}})& =& \dim H_1(\partial N_K; V_{\widehat{\al}})&=& 2,\\
\dim H^2(\partial N_K; V_{\widehat{\al}})& =& \dim H_2(\partial N_K; V_{\widehat{\al}})&=& 1.\end{array}  $$
\end{lem}

\begin{proof}
We let $\mu$ and $\la$ be the meridian and longitude of $K$.
Note that $\al(\la)$ is trivial and that $\al(\mu)$ is diagonal with eigenvalues
$z,ze^{2\pi i/n},\dots,ze^{2\pi i(n-1)/n}$, which are distinct. Note that $\al(\mu)$ has precisely one eigenvalue which equals one.
A direct calculation using Equation (\ref{eq1}) shows that $H^0(\partial N_K,V_{\widehat{\al}})=\CC$ and $H_0(\partial N_K,V_{\widehat\al})=\CC$, and
duality gives that $H_2(\partial N_K;V_{\widehat{\al}})=\CC$ and $H^2(\partial N_K; V_{\widehat{\al}})=\CC$.
Since the Euler characteristic of the torus $\partial N_K$ is zero we see that
$\dim H^1(\partial N_K; V_{\widehat{\al}}) = \dim H_1(\partial N_K; V_{\widehat{\al}})= 2$.
\end{proof}

\begin{lem}\label{lem-knot-V}
Let $K$ be a knot.
For $i=1,\dots,\ell,$ let  $\chi_i\co H_1(L_n) \to \CC^*$  be a non--trivial character and $z_i\in U(1)$ with $z_i^n=1$. Let $V = \CC^{n \ell}$ and
consider the representation $\al=\bigoplus_{i=1}^\ell \al_{(n,\chi_i,z_i)}\co \pi_1(N_K) \to \Aut(V)$. Then the following hold:
\begin{enumerate}
\item[(i)] $\dim H^0(N_K; V_\al)=0$,
\item[(ii)] if $\al$ is orthogonal or unitary with respect to a non--singular form on $V$,
then
$\dim H^1(N_K;V_\al)\geq \ell$.
\end{enumerate}
\end{lem}

\begin{proof}
The first statement is an immediate consequence of Equation (\ref{eq1}) and the assumption that $\chi_i$ are non-trivial.
By Lemma \ref{lem-torus-V} we have  $\dim H_1(\partial N_K;V_\al)=2 \ell$. Now consider the following short exact sequence
$$ H^1(N_K;V_\al) \lto H^1(\partial N_K;V_\al)\lto H^2(N_K,\partial N_K;V_\al). $$
It follows that either $\dim H^1(N_K;V_\al)\geq \ell$ or $\dim H^2(N_K,\partial N_K;V_\al)\geq \ell$.
But by Poincar\'e duality and by Lemma \ref{lem-b} the latter also equals $\dim H^1(N_K;V_\al)$.
\end{proof}

\section{Main Results}
In this section we establish the results discussed in the introduction. In \S \ref{sec-coh}, we present  cohomology arguments showing $\dim H^1(N_K; sl(n,\CC)_{\ad \al}) \geq n-1$  for any irreducible metabelian representation
$\al \co \pi_1 (N_K) \to \SL(n,\CC)$. In \S \ref{sec-dim}, we prove  that any algebraic component $X_j$
of $X_n(N_K)$ has dimension $\dim X_j \geq n-1$, in case $X_j$ contains
the character of a regular representation $\al\co\pi_1(N_K) \to \SL(n,\CC)$.
This is a generalization to $\SL(n,\CC)$ of a theorem due to Thurston for $\SL(2, \CC)$ (see \cite[Theorem 3.2.1]{CS83}).

At this point, we make the assumption that $\dim H^1(N_K; sl(n,\CC)_{\ad \al}) = n-1$.
Using this condition, we show in \S \ref{sec-simple} that every irreducible metabelian character $\xi_\al$ is a simple point of the character variety $X_n(N_K)$.
In  \S \ref{sec-SU}, we prove that every irreducible metabelian representation
$\al \co \pi_1 (N_K) \to \SL(n,\CC)$  has finite image and is conjugate to a unitary representation, and we  develop $SU(n)$ versions of the earlier results. In  \S \ref{sec-proof}, we give the proofs of Theorems \ref{thm1} and \ref{thm2}, and in \S \ref{sec-examples}, we present examples illustrating how to apply these techniques.


\subsection{Cohomology arguments} \label{sec-coh}
Assume now that
$\al \co\pi_1(N_K)\to \SL(n,\CC)$
is a representation and let $\widehat\al \co \pi_1(\partial N_K) \to \SL(n,\CC)$ denote
its restriction to the boundary torus.
Throughout much of what follows, we will assume that  $\al$ is a \emph{regular representation},
meaning that $\al$ is irreducible
 and that the image of $\widehat \al$ contains a matrix with $n$ distinct
eigenvalues. The subset of regular representations is clearly Zariski open in $R_n(N_K)$,
and every irreducible metabelian representation of $\pi_1(N_K)$
is regular (see Remark \ref{remark-regular}).

Choose $g \in \pi_1(\partial N_K)$ so that $\al(g)$ has $n$ distinct eigenvalues.
Then this matrix is diagonalizable, and any other matrix that commutes with it must
lie in the same maximal torus. Since $  \pi_1(\partial N_K) \cong \ZZ \oplus \ZZ$ is abelian,
we see that the stabilizer subgroup of $\widehat\al $ under conjugation is again this
maximal torus. From this, Poincar\'e duality and Euler characteristic considerations,
we conclude that

$$\begin{array}{rcl} \dim H^0(\partial N_K; sl(n,\CC)_{\ad \widehat{\al}})& = &n-1,\\
\dim H^1(\partial N_K; sl(n,\CC)_{\ad \widehat{\al}}) &=& 2(n-1), \text{ and }\\
\dim H^2(\partial N_K; sl(n,\CC)_{\ad \widehat{\al}}) &=& n-1.\end{array} $$

We now consider the long exact sequence in twisted cohomology associated with the
pair $(N_K, \partial N_K).$
The inclusions
$$(\partial N_K, \varnothing) \stackrel{i}{\hookrightarrow} (N_K, \varnothing)  \stackrel{j}{\hookrightarrow} (N_K, \partial N_K)$$
induce the following long exact sequence (coefficients in $sl(n,\CC)$ twisted by $\ad \al$ or $\ad \widehat{\al}$ understood).

\begin{equation} \label{eq5}
 \begin{array}{ccccccccccccccccc}
0&\lto&  H^0(N_K) &\lto& H^0(\partial N_K) &\lto& H^1(N_K, \partial N_K)&\\
&\stackrel{j^1}{\lto} & H^1(N_K) & \stackrel{i^1}{\lto}& H^1(\partial N_K) &
 {\lto}& H^2(N_K, \partial N_K) \\
 & \stackrel{j^2}{\lto}&  H^2(N_K) &\stackrel{i^2}{\lto}& H^2(\partial N_K) &\lto& H^3(N_K,\partial N_K) &\lto&0.
\end{array}
\end{equation}

Exactness of the middle row implies that
$$\dim H^1(N_K) + \dim H^2(N_K, \partial N_K) \geq \dim H^1(\partial N_K) = 2n-2,$$
and by Poincar\'e duality and Lemmas \ref{lem-b} and \ref{lem-c},  we have that $\dim H^1(N_K) = \dim H^2(N_K, \partial N_K).$
This implies $\dim H^1(N_K)  \geq n-1$.

The next proposition shows that the
${\rm image} \left( i^1\co H^1(N_K) {\lto} H^1(\partial N_K) \right)$
has dimension $n-1$, and
 this should be viewed as an instance of the following general principle.
Suppose $N$ is a 3-manifold with boundary $\partial N = \Si$ a compact Riemann surface of genus $g$.
Goldman proved that  the smooth part of the character variety
 $X_n(\Si)$ carries a natural symplectic structure \cite{Gol84}, and a folklore result implies that the image of the restriction $X_n(N) \to X_n(\Si)$ is Lagrangian.
 This idea has been made precise  by A. Sikora, who studied this in the general setting of representations into reductive Lie groups in \cite{Si09},
under the assumption that $\partial X$ is a connected
surface of genus $g \geq 2.$
We  state and prove  analogous results for $\SL(n,\CC)$ representations of
 knot complements $N_K$, which is the main case of interest here.

\begin{prop} \label{prop-Lag}
If $K$ is a knot and $\al \co\pi_1(N_K)\to \SL(n,\CC)$ is a regular representation,
then the image
$${\rm image} \left( i^1 \co H^1(N_K;sl(n,\CC)_{\ad \al}) {\lto} H^1(\partial N_K;sl(n,\CC)_{\ad \widehat\al}) \right)$$
has dimension $n-1$ and is Lagrangian with respect to the symplectic structure $\Omega$
defined below.
It follows that $$\dim H^1(N_K; sl(n,\CC)_{\ad \al}) \geq  n-1.$$
\end{prop}
\begin{proof}
 The fact that
 ${\rm image}(i^1)$
  has dimension $n-1$
follows easily from a diagram chase of the long exact sequence
 (\ref{eq5}), using the fact that $\widehat\al(\pi_1(\partial N_K))$ contains an element
 with $n$ distinct eigenvalues, hence
 $\dim H^0(\partial N_K; sl(n,\CC)_{\ad \widehat\al})=n-1=\dim H^2(\partial N_K; sl(n,\CC)_{\ad \widehat\al})$
and $\dim H^1(\partial N_K; sl(n,\CC)_{\ad \widehat\al}) =2n-2.$

 The symplectic structure $\Omega$ on $H^1(\partial N_K; sl(n,\CC))$
 is defined by composing the cup product with the symmetric bilinear pairing
 obtained by first multiplying the matrices and then taking the trace:
\begin{eqnarray*}
&sl(n,\CC) \times sl(n,\CC) \to gl(n,\CC) \to \CC& \\
&(A,B) \mapsto A \cdot B \mapsto \tr(A \cdot B).&
\end{eqnarray*}

 We have already seen that the
 ${\rm image} (i^1) $
 has dimension $n-1$,
 so we just need to show
 that it is isotropic with respect to $\Omega$.

 Suppose $x,y \in H^1(N_K; sl(n,\CC)_{\ad \al})$
  and consider the long exact sequence (\ref{eq5})
 with untwisted coefficients in $\CC$.
 Let $\smile$ denote the combined cup and matrix product,
 so $x \smile y \in H^2(N_K; gl(n,\CC)_{\ad \al})$.
Using the commutative diagram
$$ \begin{CD}
 H^1(N_K; sl(n,\CC)_{\ad \al}) \times H^1(N_K; sl(n,\CC)_{\ad \al})@>{\smile}>> H^2(N_K; gl(n,\CC)_{\ad \al}) \\
@VV{i^1 \times i^1}V  @VV{i^2}V  \\
 H^1(\partial N_K; sl(n,\CC)_{\ad \widehat\al}) \times H^1(\partial N_K; sl(n,\CC)_{\ad \widehat\al})@>{\smile}>>
 H^2(\partial N_K; gl(n,\CC)_{\ad \widehat\al}),
 \end{CD}$$
 we see that $\Omega(i^1(x), i^1(y))=\tr(i^1(x)\smile i^1(y))=  \tr i^2(x \smile y)$.
 This shows $\Om(i^1(x), i^1(y))$ lies in the image of
 \begin{equation}\label{eq6}
H^2(N_K; \CC) {\lto} H^2(\partial N_K; \CC),
 \end{equation}
 which by exactness of the third row of the long exact sequence (\ref{eq5}), now taken with
 untwisted $\CC$ coefficients,
 equals the kernel of the surjection $H^2(\partial N_K;\CC) \to H^3(N_K,\partial N_K;\CC).$
 However, it is not difficult to compute $H^3(N_K,\partial N_K;\CC) = \CC = H^2(\partial N_K;\CC),$
 and this implies that the map in Equation (\ref{eq6}) is the zero map.
   \end{proof}


\subsection{Dimension arguments} \label{sec-dim}
 In this subsection, we give a lower bound on the dimension of algebraic components of the character variety $X_n(N_K)$ containing
a regular representation.

\begin{prop} \label{prop-dim-alg-comp}
If $\al \co\pi_1(N_K)\to \SL(n,\CC)$ is a regular representation,
then any algebraic component $X_j \subset X_n(N_K)$ containing $\xi_\al$ has $\dim X_j \geq n-1.$
\end{prop}

\begin{proof}
If $\xi_\al$ is a smooth point of $X_j$, then by Proposition \ref{prop-Lag} we have $\dim X_j = \dim H^1(N_K; sl(n,\CC)_{\ad \al}) \geq n-1$. Otherwise,  we can choose $\be \co\pi_1(N_K)\to \SL(n,\CC)$ a regular representation
close to $\al$ )
so that $\xi_{\be} \in X_j$ is smooth.
Applying Proposition \ref{prop-Lag} to $\be,$ it follows that $\dim X_j = \dim H^1(N_K; sl(n,\CC)_{\ad \be}) \geq n-1$.
\end{proof}

Since every irreducible metabelian representation is regular,  we obtain the following
as a direct consequence.
\begin{cor} \label{cor-dim-alg-comp}
If $\al \co\pi_1(N_K)\to \SL(n,\CC)$ is an irreducible  metabelian representation,
then any algebraic component $X_j$ of $X_n(N_K)$ containing $\xi_\al$ has $\dim X_j \geq n-1.$
\end{cor}


\subsection{Simple points in $\boldsymbol{X_n(N_K)}$} \label{sec-simple}
This subsection presents a smoothness result for  irreducible characters which    is proved using the powerful deformation argument from \cite{HPS01}. A more detailed explanation of this beautiful argument is presented in  Appendix \ref{app-deform}, following \cite{HPS01, HP05, AHJ10}, and the original idea can be traced back to a deep theorem of Artin \cite{Ar68}.

Recall that a point $\xi \in X$ in an affine algebraic variety is called  a
\emph{simple point} if it is contained in a unique algebraic component of $X$ and is a
smooth point of that component. The next result, which essentially follows from Theorem 3.2 in \cite{AHJ10}, implies that
every irreducible metabelian character $\xi_\al$ such that $\dim H_1(N_K; sl(n,\CC)_{\ad \al}) =n-1$
 is a simple point of $X_n(N_K).$

\begin{prop} \label{prop-simple}
If $\al\co \pi_1(N_K) \to \SL(n,\CC)$ is a regular  representation
such that $\dim H^1(N_K; sl(n,\CC)_{\ad \al}) =n-1$,
then $\xi_\al$ is a simple point in the character variety $X_n(N_K).$
\end{prop}

Proposition  \ref{prop-simple}  applies to any irreducible metabelian $\SL(n,\CC)$ representation.

We give a full account of this proposition in the Appendix, and here we briefly explain the basic idea.
By irreducibility of $\al$ and Luna's \'etale slice theorem \cite{Lu73},
it follows that $\xi_\al$ is a simple point of $X_n(N_K)$
if and only if $\al$ is a simple point of $R_n(N_K).$ The same is true for $\widehat{\al},$ and the hypotheses ensure that
$\widehat\al$ is a simple point of $R_n(\partial N_K)$.
The main idea is to construct formal deformations for all (Zariski) tangent vectors and to show their integrability by using the fact that all obstructions  project faithfully under projection to $\partial N_K$, where they are known to vanish by the fact that
$\widehat\al$ is a simple point of $R_n(\partial N_K)$.


\subsection{$\boldsymbol{SU(n)}$ results} \label{sec-SU}
This subsection contains the $SU(n)$ analogues of the earlier results on irreducible metabelian  representations.
We will prove   that any irreducible metabelian representation $\al \co\pi_1(N_K)\to \SL(n,\CC)$
satisfying the condition  $\dim H^1(N_K; sl(n,\CC)_{\ad \al}) =  n-1$ has finite image and is therefore conjugate to a unitary representation.

We begin with a few general observations.
 If $\pi$ is a finitely generated group and $\al \co \pi \to SU(n)$ is a representation, then we obtain
an $\SL(n,\CC)$ representation by composing $\al$ with the  inclusion $SU(n) \subset \SL(n,\CC)$.
Irreducibility of $\al$ is preserved under this correspondence, and the map
$R_{SU(n)}(\pi) \to R_n(\pi)$  descends to a well-defined injective map
$X_{SU(n)}(\pi) \lto X_n(\pi)$ between the two character varieties.
 Here and in the following, we set $R_{SU(n)}(\pi) = \Hom(\pi, SU(n))$ and
use $X_{SU(n)}(\pi)$ to denote the character variety of $SU(n)$ representations of $\pi.$

On the level of Lie algebras,
 the complex Lie algebra $sl(n,\CC)$ is obtained by tensoring the real Lie algebra $su(n)$ with $\CC$, i.e. we have
$$sl(n,\CC) \cong su(n) \otimes \CC.$$
Thus, for  $\al \co \pi \to SU(n)$, we see that for any $i \geq 0$ we have
\begin{equation} \label{cohomology-SU}
H^i(\pi; sl(n,\CC)_{\ad\al}) \cong H^i(\pi;su(n)_{\ad \al}) \otimes \CC.
\end{equation}

In the following proposition, we use $L_n$ to denote the $n$--fold  branched
cover of $\Si^3$ branched along $K$.
\begin{prop} \label{prop-finite}
Suppose $\al \co\pi_1(N_K)\to \SL(n,\CC)$ is an irreducible metabelian representation.
If $\dim H^1(N_K; sl(n,\CC)_{\ad \al}) = n-1$, then
$H_1(L_n)$ is finite. In particular, $\al$ has finite image and is conjugate to a
unitary representation.
\end{prop}

\begin{proof}
It follows from Theorem \ref{thm-BF} that we can assume that $\al=\al_{(n,\chi)}$ for some character
 $\chi \co H_1(L_n)\to \CC^*$.
Let $\th_1\co \pi_1(N_K)\to \GL(1,\CC)$ be the trivial representation
and $\al_n\co \pi_1(N_K)\to \Aut(\CC[\ZZ/n])$ be the regular representation corresponding to the canonical projection map $\pi_1(N_K)\to \ZZ\to \ZZ/n$.
 By Lemma \ref{lem-a} we have
  the following isomorphism of representations:
$$ \ad \al \oplus \th_1\cong \al_n \oplus \bigoplus\limits_{i=1}^{n-1} \be_{(n,\chi_i)},$$
 where  $\chi_1, \ldots, \chi_{n-1}$ are characters.
 Clearly $\theta_1$ and $\al_n$ are orthogonal representations.
 Furthermore by Lemma \ref{lem-c} the representation  $\ad \al \co \pi_1(N_K) \to \Aut(sl(n,\CC))$ is an isometry with respect to the Killing form.
If we equip $sl(n,\CC)$ with the standard basis and we thus view $\ad \al$ as a representation to $gl(n^2-1,\CC)$,
then it follows from the definition of the Killing form, that $\ad \al$ is an orthogonal representation.
It now follows that $\be:=\bigoplus_{i=1}^{n-1} \be_{(n,\chi_i)} \co \pi_1(N_K)\to \GL(n(n-1),\CC)$ is also orthogonal.
By Lemma \ref{lem-knot-V} we now have
$$ \begin{array}{rcl} \dim H^1(N_K; sl(n,\CC)_{\ad \al})&=&\dim H_1(N_K;\CC[\ZZ/n])-1+\dim H^1(N_K;\CC_{\be}^{n(n-1)})\\
&=&b_1(L_n)+1-1+\dim H^1(N_K;\CC^{n(n-1)}_\be)\\
&\geq &b_1(L_n) + n-1.\end{array}$$
The condition $\dim H^1(N_K; sl(n,\CC)_{\ad \al}) = n-1$ now shows that $b_1(L_n)=0.$
Thus $H_1(L_n) = H_1(N_K;\ZZ[t^{\pm 1}])/(t^n-1)$ is finite, and this implies
$\al$ has finite image and is conjugate to a unitary representation.
\end{proof}

Proposition \ref{prop-finite} implies that metabelian representations $\al \co\pi_1(N_K)\to \SL(n,\CC)$ are often conjugate to unitary representations, and for that reason we  develop
$SU(n)$ versions of the previous results.
As the proofs are similar to those already given, we leave the details to the industrious reader.

We begin with the $SU(n)$ version of Proposition \ref{prop-Lag}.
Just as in the $\SL(n,\CC)$ case, we say a representation $\al \co \pi_1(N_K) \to SU(n)$ is   \emph{regular}
if it is irreducible and if the image of the restriction $\widehat{\al} \co\pi_1(\partial N_K)\to SU(n)$
contains a matrix with $n$ distinct eigenvalues.

Note that the definition of the symplectic form $\Omega$ on $H^1(\partial N_K; sl (n,\CC)_{\ad \widehat\al})$ in the proof of Proposition  \ref{prop-Lag}   carries over easily to the $SU(n)$ setting, and we
use $\Omega_{SU(n)}$ to denote the resulting
symplectic form on $H^1(\partial N_K; su(n)_{\ad \widehat\al})$.

\begin{prop} \label{prop-Lag-SU}
If $K$ is a knot and $\al \co\pi_1(N_K)\to SU(n)$ is a regular representation,
then the image
$${\rm image} \left( i^1 \co H^1(N_K;su(n)_{\ad \al}) {\lto} H^1(\partial N_K;su(n)_{\ad \widehat\al}) \right)$$
has real dimension $n-1$ and is Lagrangian with respect to the natural symplectic structure $\Omega_{SU(n)}$.
It follows that $$\dim_\RR H^1(N_K; su(n)_{\ad \al}) \geq  n-1.$$
\end{prop}

Next, we present the $SU(n)$ version of Proposition \ref{prop-dim-alg-comp}.
 Recall that $X_{SU(n)}(N_K)$ is a real algebraic variety.

\begin{prop} \label{prop-dim-alg-comp-SU}
If $\al \co\pi_1(N_K)\to SU(n)$ is a regular  representation,
then any algebraic component $X_j \subset X_{SU(n)}(N_K)$ containing $\xi_\al$ satisfies $\dim_\RR X_j \geq n-1.$
\end{prop}

Since all irreducible metabelian $SU(n)$ representations are regular,  Proposition \ref{prop-dim-alg-comp-SU} applies to give the following as a direct consequence.
\begin{cor} \label{cor-dim-alg-comp-SU}
If $\al \co\pi_1(N_K)\to SU(n)$ is an irreducible  metabelian representation,
then any algebraic component $X_j$ of  $X_{SU(n)}(N_K)$ containing $\xi_\al$ has $\dim_\RR X_j \geq n-1.$
\end{cor}

The final result  is an $SU(n)$ version of Proposition \ref{prop-simple}.

\begin{prop} \label{prop18}
If $\al\co \pi_1(N_K) \to SU(n)$ is a regular representation
such that $\dim_\RR H_1(N_K; su(n)_{\ad \al}) =n-1$,
then $\xi_\al$ is a simple point in the character variety $X_{SU(n)}(N_K).$
\end{prop}


\subsection{Proofs of Theorems \ref{thm1} and \ref{thm2}} \label{sec-proof}
In this subsection, we   prove the two main results from the Introduction.
\begin{proof}[Proof~of~Theorem~\ref{thm1}]
Suppose $\al$ is an irreducible metabelian representation with
$\dim H^1(N_K; sl(n,\CC)_{\ad \al}) =n-1.$
Applying
Proposition \ref{prop-finite}, we see that $\al$ has finite image and
hence is conjugate to a unitary representation.
Since $\al(\mu)$ has $n$ distinct eigenvalues, Proposition \ref{prop-simple} applies and
gives rise to a smooth complex $(n-1)$--dimensional family of $\SL(n,\CC)$ characters
near $\xi_\al \in X_n(N_K)$.

Conjugating, if necessary, we can arrange that  $\al$ is unitary.
In that case, Equation (\ref{cohomology-SU}) implies that
$$H^1(N_K; sl(n,\CC)_{\ad \al}) = H^1(N_K; su(n)_{\ad \al}) \otimes \CC,$$
and it follows that $\dim_\RR H^1(N_K; su(n)_{\ad \al}) =n-1.$
Thus Proposition \ref{prop18} applies and
gives rise to a  smooth real $(n-1)$--dimensional family of irreducible characters
near $\xi_\al \in X_{SU(n)}(N_K)$.

Note that Proposition \ref{prop-finite} shows that $b_1(L_n)=0$, and thus
every irreducible metabelian representation $\be\co \pi_1(N_K) \to \SL(n,\CC)$
factors through a finite group. In particular, this shows that up to conjugation there are only finitely many
irreducible metabelian $\SL(n,\CC)$ representations, and their characters give rise to
a finite collection of points in the character variety $X_n^*(N_K).$ It follows that
we can take either of the two deformation families of conjugacy classes of irreducible
representations so that $\xi_\al$ is the unique metabelian representation within the family.
\end{proof}

\begin{proof}[Proof~of~Theorem~\ref{thm2}]
Let $\al \co \pi_1(N_K)\to \SL(n,\CC)$ an irreducible metabelian representation, and
$\varphi \co \pi_1(N_K)\to \ZZ/n\ltimes H$  a homomorphism such that $\al$ factors through $\varphi$ and with $H$ finite. Set $k=|H|$.

We first consider the cover $p\co\widetilde{N}_\varphi \to N_K$  corresponding to $\varphi$.
Note that there exist precisely $k=|H|$ characters $H\to U(1)$.
We denote this set by $\{\si_1,\dots,\si_k\}$, where we assume that $\si_1$ is the trivial character.
It is not difficult to see that the representation $\si_1\oplus \dots \oplus \si_k\co H\to \Aut(\CC^k)$ is isomorphic to the regular representation
$H\to \Aut(\CC[H])$.
We denote the representation $\pi_1(N_K)\to \Aut(\CC[\ZZ/n\ltimes H])$ by $\varphi$ as well.
Then it is straightforward to verify that
$$ \varphi\cong \bigoplus\limits_{i=1}^k \be_{(n,\si_i)}.$$
In particular, setting $V=\CC^{kn}$ and $U=\CC^{n},$ we have
$$ b_1(\widetilde{N}_\varphi)=b_1(N;V_\varphi)=\sum\limits_{i=1}^k b_1(N;U_{\be_{(n,\si_i)}} ).$$
Note that each $\be_{(n,\si_i)}$ is a unitary representation.
It now follows immediately from Lemma \ref{lem-knot-V} that $b_1(\widetilde{N}_\varphi)\geq k$.
Furthermore, if
$b_1(\widetilde{N}_\varphi)=k$ then it  follows  that $b_1(N;U_{\be_{(n,\si_i)}})=1$ for each $i=1, \ldots , k$.
Statement (i) now follows immediately from
Lemma \ref{lem-a}.

We now turn to the proof of (ii).
We write $T=\partial N_K$.
Note that the image of the restriction $\widehat \varphi \co \pi_1(T)\to \ZZ/n\ltimes H$ has order $n$. In particular
the preimage of $T$ under the covering
$p\co \widetilde{N}_\varphi\to N_K$
has $k=|H|$ components, and we denote them by $T_1,\dots,T_k$.
Note that in each $T_i$ there exist  simple closed curves $\mu_i$ and $\la_i$ such that $p|_{\mu_i}$ restricts to
an $n$--fold cover of the meridian $\mu\subset T$ and such that $p|_{\la_i}$ restricts to a homeomorphism with the longitude $\la$ of $T$.
Note that $\mu_i,\la_i$ form a basis for $H_1(T_i)$.

We now denote
by $\widehat{\Si}_\varphi$ the result of gluing $k$ solid tori $S_1,\dots,S_k$ to the boundary of $\widetilde{N}_\varphi$ such that
each $\mu_i$ bounds a disk in $S_i$. The projection map $p\co \widetilde{N}_\varphi\to N_K$
then extends in a canonical way to a covering map $\widehat{\Si}_\varphi\to \Si$, branched over $K$, and that proves (ii).

We finally turn to the proof of (iii).
Consider the following Mayer--Vietoris sequence:
$$ \bigoplus\limits_{i=1}^k H_1(T_i)\to \bigoplus\limits_{i=1}^k H_1(S_i)\, \oplus \, H_1(\widetilde{N}_\varphi)\to H_1(\widehat{\Si}_\varphi)\to 0.$$
It follows immediately that
$$ b_1(\widehat{\Si}_\varphi)\geq k+b_1(\widetilde{N}_\varphi)-2k=b_1(\widetilde{N}_\varphi)-k.$$
In particular if $b_1(\widehat{\Si}_\varphi)=0$, then
$b_1(\widetilde{N}_\varphi)\leq k$. Applying (i) shows we have equality here and that (iii) holds.
\end{proof}

\subsection{Examples} \label{sec-examples}
In this subsection, we show how to construct deformations of   metabelian representations
$\al \co \pi_1(N_K) \to \SL(n,\CC)$ in specific situations.

We begin with some general comments about the rank two case.
As mentioned in the introduction, by results of Culler and Shalen \cite{CS83}, if $K$ is a small knot,
then any irreducible metabelian representation
$\al \co \pi_1(N_K) \to \SL(2,\CC)$ lies on an algebraic component of $X_2(N_K)$ of dimension one.
Since all torus knots and all two-bridge knots are small, this tells us that Theorem \ref{thm1} applies to many knots
in rank two. Interestingly, not all such knots admit irreducible metabelian $\SL(2,\CC)$ representations.
For example, in the notation of Rolfsen's table \cite{Ro76}, this occurs for the knots $10_{124}$ and $10_{153}$.
Note that $10_{124}$ is the $(3,5)$--torus knot and is a fibered knot of genus 3, whereas $10_{153}$
is not a torus knot but it is fibered of genus 4.
A simple calculation using \cite[Theorem 3.7]{BF08} shows that both knots admit irreducible metabelian  representations in $\SL(3,\CC)$ and $\SL(5,\CC)$, indeed up to conjugation $10_{124}$ admits 8 such  representations in rank 3 and 16 in rank 5, whereas $10_{153}$ admits 16
such representations in rank 3 and 24 in rank 5. In both cases, we see that $H_1(L_3)$ and $H_1(L_5)$ are finite, and so the irreducible metabelian characters are isolated points in the character variety of the metabelian quotient $\pi_1(N_K)/ \pi_1(N_K)^{(2)}.$
Proposition \ref{prop-Lag} applies to show they can be deformed to nearby non-metabelian irreducible representations.

We now investigate situations to which Theorem \ref{thm1} applies, and for that purpose we will consider a fibered knot $K$ of genus one
in an integer homology 3-sphere $\Si$. Note that by the proof of Proposition 5.14 of \cite{BZ85}, it follows that the complement
$\Si \sm \tau(K)$ is homeomorphic to that of
 the trefoil or the figure eight knot. The trefoil knot has irreducible metabelian representations only in rank $2, 3,$ and $6$.  The figure eight knot, on the other hand, has irreducible metabelian $\SL(n,\CC)$ representations for all but finitely many ranks, which follows directly from Theorem \ref{thm:infinitelymanyreps}.
Indeed, the number of conjugacy classes of irreducible metabelian representations for both knots
$\al \co \pi_1(N_K) \to SU(n)$ can be determined in terms of the orders $|H_1(L_k)|$ taken over all divisors $k$ of  $n$, and direct computation shows that
the trefoil has a unique irreducible metabelian representation in ranks 2 and 3, whereas
  the figure eight has increasingly many as the rank $n \to \infty$. Applying
  Theorem 3.7 of \cite{BF08}, we compute  the number of distinct
conjugacy classes of irreducible metabelian $SU(n)$ representations for the figure eight knot, and the results for $1\leq n \leq 21$
are listed in Table \ref{figure-eight}.

The next result shows that any algebraic component of $X_n(N_K)$ containing such a representation has dimension $n-1$.
Thus Theorem \ref{thm1} applies and gives a nice local description of the character variety near these metabelian characters.

 \begin{table}[t]
   \begin{tabular}{|c|r| |c|r| |c|r|}
\hline  n & $\#$ & n  & $\#$  & n  & $\#$ \\ \hline
1& 1 & 8 & 270  &  15& 124,024 \\
2& 2  & 9& 640  & 16& 304,290  \\
3& 5  & 10& 1500 &  17& 750,120\\
4& 10  & 11 & 3600   & 18& 1,854,400 \\
5& 24  & 12& 8610 & 19& 4,600,200 \\
6& 50  & 13& 20880  &  20 &  11,440,548  \\
\;\;7\;\;& \;\;120  & \;\;14\;\;& \;\;50700 &  \;\;21\;\; & \;\; 28,527,320 \\ \hline
  \end{tabular} \medskip
 \caption{The number of conjugacy classes of irreducible metabelian $SU(n)$ representations for the figure eight knot for $1 \leq n \leq 21$}
 \label{figure-eight}
 \end{table}

\begin{prop} \label{prop-fiber-one}
Suppose $K$ is a fibered knot of genus one in a homology 3-sphere $\Si$ whose $n$--fold branched cover has
$H_1(L_n)$ finite and $\al \co \pi_1(N_K) \to \SL(n,\CC)$ is an irreducible
metabelian representation. Then any algebraic component $X_j$ of $X_n(N_K)$ containing $\xi_\al$
has $\dim X_j = n-1.$
\end{prop}

\begin{remark}
As mentioned above, if $K$ is a genus one fibered knot in an integral homology 3-sphere $\Si$, thenthe complement
$\Si \sm \tau(K)$ is homeomorphic to that of the trefoil or figure eight knot (see Proposition 5.14 of \cite{BZ85}).
\end{remark}

\begin{proof}
By Proposition \ref{prop-Lag}, we have that $\dim X_j \geq n-1,$ so it is enough to show
$\dim X_j \leq n-1.$ If $R_j$ is the algebraic component of $R_n(N_K)$ lying above $X_j,$
then
we will show that $\dim R_j \leq n^2+n-2.$ This is sufficient because we know that  $R_j$ contains the irreducible
representation $\al$, and so the generic fiber of the quotient map $t \co R_j \to X_j$
has dimension $n^2-1.$

Consider the subset of $R_j$ defined by
$$\widehat{R}_j = \{ \varrho \in R_j \mid \varrho \text{ is irreducible and $\varrho(\mu)$ has $n$ distinct eigenvalues} \}.$$
This is obviously a Zariski open subset, and since $\al \in \widehat{R}_j,$  it is nonempty.
In particular, we see that $\dim \widehat{R}_j = \dim R_j.$

Given $A \in \SL(n,\CC)$, we use $\Phi_A(t) = \det(tI - A)$ to denote its characteristic polynomial.
In general, given $\ga \in \pi$, the association $\varrho \mapsto \Phi_{\varrho(\ga)}(t)$ gives  an algebraic
 map $\Phi_{\cdot(\ga)} \co R_n(\pi) \lto \CC^{n-1}$, where
$\Phi_{\varrho(\ga)}(t) = t^n + a_1 t^{n-1} + \cdots + a_{n-1} t +(-1)^n$ gives the point $(a_1,\ldots, a_{n-1}) \in \CC^{n-1}.$

 Taking $\ga =\la,$ the longitude of $K$, we define
$$Z_j = \{ \varrho \in \widehat{R}_j \mid  \Phi_{\varrho(\la)}(t) = (t-1)^n \}.$$
Clearly $Z_j$ is a Zariski closed subset of $\widehat{R}_j$. Since $Z_j$ is obtained by applying $n-1$ algebraic equations,
we see that $\dim Z_j \geq \dim R_j -(n-1)$ (see \cite[p. 75, Corollary 2]{Sh95}). Furthermore, since $\al \in Z_j$, we see that $Z_j$ is nonempty.

For any fibered knot, the commutator subgroup of $\pi_1(N_K)$ is the
finitely generated free group given by the fundamental group of the fiber.
In the case of a fibered knot of genus one, this group is a free group of rank two, and
 we obtain the short exact sequence
\begin{equation} \label{short-exact}
1 \to F_2  \lto \pi_1(N_K) \lto \ZZ \to 1.
\end{equation}
Taking $S$ to be the fiber surface, then $F_2 = \pi_1(S) = \langle a,b \rangle$ and we can write the longitude as
$\la = aba^{-1}b^{-1}.$ Thus, a given
representation $\varrho \co \pi_1(N_K) \to \SL(n,\CC)$ is metabelian
if and only if its restriction to $F_2$ is abelian, namely if $\varrho(a)$ and $\varrho(b)$ commute.
Since $\la = [a,b],$ we see that $\varrho$ is metabelian if and only if $\varrho(\la) =I.$
This shows that every irreducible metabelian representation in $\widehat{R}_j$ is contained
in $Z_j$, and we will now show the reverse inclusion.

Suppose $\varrho  \in Z_j$. Then since $Z_j \subset \widehat{R}_j$, $\varrho$ is irreducible and $\varrho(\mu)$ has  $n$ distinct eigenvalues.
Thus $\varrho(\mu)$ is contained in a unique maximal torus, which we can arrange by conjugation to be the standard
maximal torus of diagonal matrices in $\SL(n,\CC)$. Since $\varrho(\la)$ commutes with $\varrho(\mu)$, it follows that
$\varrho(\la)$ is also a diagonal matrix.
The condition  that $\Phi_{\varrho(\la)}(t)= (t-1)^n$ implies
 $\varrho(\la) =I$, and the sequence (\ref{short-exact}) shows that $\varrho$ is necessarily metabelian.

We now make use of the assumption that $H_1(L_n)$ is finite. This implies that, up to conjugation, there are only finitely many irreducible
metabelian representations. Thus the quotient of $Z_j$ by conjugation is a finite collection of points, and since
every $\varrho \in Z_j$ is also irreducible, we conclude that $\dim Z_j = n^2-1.$
Using that $$\dim R_j = \dim \widehat{R}_j \leq \dim Z_j +(n-1) = n^2+n-2,$$ we conclude that  $\dim X_j \leq n-1.$
\end{proof}

Proposition \ref{prop-fiber-one} applies to irreducible metabelian representations of the figure eight knot in all ranks
(see Table \ref{figure-eight}), but it only applies to the trefoil in ranks 2 and 3.
The only other rank where the trefoil admits irreducible metabelian representations is rank 6, and in
that case $H_1(L_6)$ is not finite.

We investigate the general situation of torus knots, and we
note that as a consequence of Proposition 3.10 (iii) of \cite{BF08}, a $(p,q)$ torus
knot $K$ has no irreducible metabelian $\SL(n,\CC)$ representations if $n$ is relatively prime to $p$ and $q$.
Torus knot groups have the following well-known presentation:
\begin{equation} \label{torus-pres}
\pi_1 (N_K)  = \langle x,y \mid x^p=y^q \rangle,
\end{equation}
where the meridian and longitude $\mu$ and $\la$  are represented
    by
    $   \mu = x^s y^r$  and $\la = x^p ( \mu )^{-pq},$ for $r,s \in \ZZ$ with $rp+sq = 1.$
We  choose $n$ to be a divisor of $q$ and work with $SU(n)$
representations for convenience.
Then any  irreducible metabelian representation $\varrho \co \pi_1 (N_K) \to SU(n)$
will satisfy $\varrho(\mu)^n=I$ and $\varrho(\la)=I,$ and this implies that
$\varrho(x)$ and $\varrho(y)$ are $p$-th and $q$-th roots of unity, respectively.

Since $\varrho(x)$ and $\varrho(y)$ are diagonalizable,  we can arrange
that
$\varrho(x)$ is conjugate to $A= {\rm diag}(a_1, \ldots, a_n)$ and
$\varrho(y)$ is conjugate to $B= {\rm diag}(b_1, \ldots, b_n)$,
where $a_1,\ldots, a_n$ are $p$-th roots of unity and $b_1,\ldots, b_n$ are $q$-th roots of unity.
Let $C_A$ and $C_B$ denote the conjugacy classes in $SU(n)$ of $A$ and $B$, respectively. The eigenspaces of $A$ and $B$ determine
partitions $(\al_1,\ldots, \al_k)$   and
  $(\be_1,\ldots, \be_\ell)$ of $n$, respectively,
 and we have
  $$\dim C_A = n^2 - (\al_1^2 + \cdots + \al_k^2) \quad \text{and}\quad  \dim C_B = n^2 - (\be_1^2 + \cdots + \be_\ell^2).$$
For instance, if $A$ has $n$ distinct eigenvalues, then $\dim C_A = n^2-n.$
In general, $\dim C_A$ and $\dim C_B$ are even numbers between  $0$ and $n^2-n.$

The component $R_j$ of $R_{SU(n)}(N_K))$ containing $\varrho$ is just the direct product
$C_A \times C_B$, and it follows that
$\dim R_j =    \dim C_A  + \dim C_B.$
If  $A$ and $B$ can be chosen so that $ \dim C_A  + \dim C_B = n^2+n-2$, then we will
be able to apply Theorem \ref{thm1}.
This will occur if say
$A$ has $n$ distinct eigenvalues and $B$ has one eigenvalue of multiplicity 1 and a second eigenvalue of
multiplicity $n-1.$ Assuming that $R_j$ contains an irreducible representation,
then just as in the proof of Proposition \ref{prop-fiber-one}, it follows that if $X_j \subset X_{SU(n)}(N_K)$ is the
quotient
of $R_j$ under conjugation, then $\dim X_j = n-1.$

For specific examples, consider the torus knots $K=T(2,q),$ where $q$ is a multiple of 3. Then direct calculation shows that any
irreducible metabelian representation $\varrho: \pi_1 (N_K) \to SU(3)$ has
$\dim H^1(N_K; su(3)_{\ad\varrho}) = 2$ (see Proposition 3.1 of \cite{BHK05}, for example).
 Hence Theorem \ref{thm1} applies to establish the existence of
2--dimensional deformation families in  $X_{SU(3)}(N_K)$ and $X_3(N_K).$
\appendix
\section{Deformation arguments} \label{app-deform}
In this appendix, we present the deformation arguments that prove Proposition \ref{prop-simple}. This material is
included for the reader's convenience.
The original arguments  were given for $\SL(2,\CC)$ and $\PSL(2,\CC)$ in \cite{HPS01} and \cite{HP05}, and they were generalized to $\SL(n,\CC)$ in \cite{AHJ10}. In what follows, we present detailed arguments for $\SL(n,\CC)$,  focusing on the implications  for the character variety $X_n(N_K)$, where  $N_K = \Si \sm \tau(K)$ is the complement of a knot in an integral homology 3-sphere.

\begin{proof}[Proof~of~Proposition~\ref{prop-simple}]
The first step is to show that $\xi_{\widehat{\al}}$ is a simple point in
$X_n(\partial N_K)$.
We do this by
comparing the dimension of the cocycles $$Z^1(\pi_1(\partial N_K); sl(n,\CC)_{\ad \al})$$
with the local dimension of $R_n(\partial N_K)$ at $\al$, which is defined to be
the maximal dimension of the irreducible components of $R_n(\partial N_K)$
containing $\al$.

First, some notation. Given a finitely generated group $\pi$ and a representation $\al\co \pi \to \SL(n,\CC)$,
let  $H^*(\pi;sl(n,\CC)_{\ad \al})$ denote the cohomology of the group
with coefficients in the $\pi$--module by $sl(n,\CC)_{\ad \al}$.

In \cite{We64}, Weil observed that there is a natural inclusion of the Zariski tangent space
  $T^{\rm Zar}_\al(R_n(\pi)) \hookrightarrow Z^1(\pi; sl(n,\CC)_{\ad \al})$
  into the space of cocycles, and we will combine  this observation with
  computations of the twisted cohomology of $\pi_1(\partial N_K)$ and
  $\pi_1(N_K)$.

Because $\partial N_K$ is a $K(\ZZ \oplus \ZZ, 1)$, we have isomorphisms
$$H^*(\partial N_K; sl(n,\CC)_{\ad \al}) \lto H^*(\pi_1(\partial N_K); sl(n,\CC)_{\ad \al}),$$
  and the inclusion $N_K \hookrightarrow  K( \pi_1(N_K), 1)$
induces maps
$$H^i(\pi_1(N_K);sl(n,\CC)_{\ad \al})\lto H^i(N_K ;sl(n,\CC)_{\ad \al})$$
that are isomorphisms when $i=0$ and $1$ and injective when $i=2$
(see \cite[Lemma 3.1]{HP05}).

Consider the 2-torus $\partial N_K$ with its standard CW--structure consisting of one 0--cell, two 1--cells and one 2--cell. It is straightforward to verify that  the spaces of twisted 1-coboundaries and 1-cocycles satisfy
\begin{eqnarray*}
\dim B^1(\partial N_K; sl(n,\CC)_{\ad \widehat{\al}}) &=& n^2-1-(n-1) = n^2-n, \text{ and} \\
 \dim Z^1(\partial N_K; sl(n,\CC)_{\ad \widehat{\al}}) &=& 2(n-1)+n^2-n = n^2+n-2.
 \end{eqnarray*}
Since $\widehat{\al}$ sits on an $(n^2+n-2)$--dimensional component,  its local dimension is
$$\dim_{\widehat\al} R_n(\partial N_K) = n^2+n-2.$$

For arbitrary $\si \in R_n(\partial N_K),$ we have
$$\dim_\si R_n(\partial N_K)
\leq \dim T^{\rm Zar}_\si(R_n(\partial N_K))
\leq \dim Z^1(\partial N_K; sl(n,\CC)_{\ad \si}).$$
In our case,
we have equality throughout, and it follows that $\widehat\al$ lies on a unique irreducible component of
$R_n(\partial N_K)$ and is a smooth point of that component (see \cite[\S 2, Theorem 6]{Sh95}).
This shows $\widehat\al$ is a simple point of $R_n(\partial N_K).$

The next step is to show that $\xi_\al$ is a simple point of $X_n(N_K).$
Consider the long exact sequence (\ref{eq5}) in cohomology associated with the
pair $(N_K, \partial N_K)$. Irreducibility of $\al$ implies that
$H^0(N_K; sl(n,\CC)_{\ad \al}) = 0$, and
Lemmas \ref{lem-b}, \ref{lem-c} and Poincar\'e duality
give $H^3(N_K, \partial N_K; sl(n,\CC)_{\ad \al}) = 0$.
Since
$H^1(N_K; sl(n,\CC)_{\ad \al}) = \CC^{n-1}$ by hypothesis,
we see  $H^2(N_K, \partial N_K; sl(n,\CC)_{\ad \al}) =  \CC^{n-1}$
by Poincar\'e duality.

Since $H^1(\partial N_K; sl(n,\CC)_{\ad \widehat\al}) = \CC^{2(n-1)}$, it follows
that the middle row of (\ref{eq5})
 $$   0 \lto  H^1(N_K) \lto H^1(\partial N_K) \lto H^2(N_K, \partial N_K) \lto 0,$$
is short exact (with coefficients in $sl(n,\CC)$ twisted by $\ad \al$ or $\ad \widehat{\al}$ understood).
Thus $j^1=0$ and $j^2=0,$ and further $i^1$ is injective and $i^2$ is an isomorphism.

We now explain the powerful
technique for deforming representations. It involves the following three steps:
\begin{itemize}
\item[(i)] constructing formal deformations,
\item[(ii)] proving integrability by showing  an infinite sequence of obstructions vanish,
\item[(iii)] proving convergence by applying a deep result of Artin \cite{Ar68}.
\end{itemize}

A \emph{formal deformation}
 of $\al$  is a homomorphism
$\al_\infty \co \pi \to \SL(n, \CC[[t]])$ given by
$$\al_\infty(g) = \exp\left(\sum_{i=1}^\infty t^i a_i(g) \right) \al(g),$$
such that $p_0(\al_\infty) = \al$,
where $p_0\co  \SL(n, \CC[[t]]) \to  \SL(n, \CC)$ is the homomorphism given
by setting $t=0$ and where $a_i \co \pi \to sl(n,\CC)_{\ad \al}, i=1,\dots,$ are 1-cochains with twisted coefficients. By \cite[Lemma 3.3]{HPS01}, it follows
that $\al_\infty$ is a homomorphism if and only
if $a_1 \in Z^1(\pi; sl(n,\CC)_{\ad \al})$ is a cocycle, and
we call an element  $a \in Z^1(\pi; sl(n,\CC)_{\ad \al})$
\emph{formally integrable} if there is a formal deformation with leading term $a_1 = a.$

Let $a_1, \ldots, a_k \in C^1(\pi; sl(n,\CC)_{\ad \al})$ be cochains such that
$$\al_k(g) =  \exp\left(\sum_{i=1}^k t^i a_i(g) \right) \al(g)$$
is a homomorphism into $\SL(n, \CC[[t]])$ modulo $t^{k+1}.$
Here, $\al_k$ is called a formal deformation of order $k$, and in this case
by \cite[Proposition 3.1]{HPS01}
there exists an obstruction class $\om_{k+1}:=\om_{k+1}^{(a_1,\ldots, a_k)} \in H^2(\pi;sl(n,\CC)_{\ad \al})$
with the following properties:

\begin{enumerate}
\item[(1)]
There is a cochain $a_{k+1} \co \pi \to sl(n,\CC)$ such that:
$$\al_{k+1}(g) = \exp\left(\sum_{i=1}^{k+1} t^i a_i(g) \right) \al(g)$$
is a homomorphism modulo $t^{k+2}$ if and only if $\om_{k+1}=0.$
\item[(2)]
The obstruction $\om_{k+1}$ is natural, i.e. if $\varphi\co \pi' \to \pi$
is a homomorphism then $\varphi^* \om_k:= \al_k \circ \varphi$ is also a
homomorphism modulo $t^{k+1}$ and $\varphi^*(\om_{k+1}^{(a_1,\ldots, a_k)})
= \om_{k+1}^{(\varphi^* a_1,\ldots, \varphi^* a_k)}.$
\end{enumerate}

\begin{lem} \label{lem-int}
Let $\al \co \pi_1(N_K) \to \SL(n,\CC)$ be an irreducible representation such that
$\dim H^1(N_K;sl(n,\CC)_{\ad \al}) = n-1$. If the image of the restriction
$\widehat\al \co \pi_1(\partial N_K) \to \SL(n,\CC)$
contains an element with $n$ distinct eigenvalues, then every cocycle
$a \in Z^1(\pi_1(N_K);sl(n,\CC)_{\ad\al})$ is integrable.
\end{lem}

\begin{proof}
Consider first the commutative diagram:

$$\begin{CD}
  H^2(\pi_1(N_K); sl(n,\CC)_{\ad \al})@>{i^*}>>H^2(\pi_1(\partial N_K); sl(n,\CC)_{\ad \widehat\al}) \\
  @VVV  @VV{\cong}V \\
    H^2(N_K; sl(n,\CC)_{\ad \al}) @>{\cong}>> H^2(\partial N_K; sl(n,\CC)_{\ad \widehat\al}).\\
    \end{CD}$$

Here, the horizontal isomorphism on the bottom follows by consideration of the long exact sequence
(\ref{eq5}), and the vertical isomorphism on the right follows since $\partial N_K$
is a $K(\ZZ \oplus \ZZ, 1).$ Further, by \cite[Lemma 3.3]{HP05}, we know the vertical map on the left is an injection, and this shows $i^*$ is an injection.

We now explain how to prove that every element $a \in Z^1(\pi_1(N_K);sl(n,\CC)_{\ad\al})$ is integrable.
Suppose (by induction) that  $a_1, \ldots, a_k \in C^1(\pi; sl(n,\CC)_{\ad \al})$
are given so that
$$\al_{k}(g) = \exp\left(\sum_{i=1}^{k} t^i a_i(g) \right) \al(g)$$
is a homomorphism modulo $t^{k+1}$.
Then the restriction $\widehat{\al}_k \co \pi_1(\partial N_K) \to \SL(n,\CC[[t]])$
is also a formal deformation of order $k$.
On the other hand, $\widehat\al_k$ is a smooth point of $R_n(\partial N_K)$,
hence by \cite[Lemma 3.7]{HPS01}, $\widehat\al_k$ extends to a formal
deformation of order $k+1$.
Therefore
$$0= \om_{k+1}^{(i^*a_1, \ldots, i^*a_k)} = i^*\om_{k+1}^{(a_1,\ldots, a_k)}.$$
As $i^*$ is injective, the obstruction vanishes, and this completes the proof of the lemma.
\end{proof}

We are now ready to conclude the proof of
Proposition \ref{prop-simple}.
Lemma \ref{lem-int}
shows that all cocycles in $Z^1(\pi_1(N_K);sl(n,\CC)_{\ad\al})$ are integrable.
Applying Artin's theorem \cite{Ar68}, we obtain from a formal deformation
of $\al$ a convergent deformation (see \cite[Lemma 3.6]{HPS01}). Thus
$\al$ is a smooth point of $R_n(N_K)$ with local dimension $\dim_\al R_n(N_K) = n-1.$
It follows that $\al$ is a simple point of $R_n(N_K)$ and this together with irreducibility of $\al$
imply that $\xi_\al$ is a simple point of $X_n(N_K).$
\end{proof}


\begin{thebibliography}{100000}

\bibitem[AHJ10]{AHJ10}
L. B. Abdelghani, M. Heusener, and H. Jebali,
{\em Deformations of metabelian representations of knot groups into $SL(3,\mathbb{C})$},
J. Knot Theory Ramific {\bf 19}  (2010) 385--404.

\bibitem[Ar68]{Ar68}
M. Artin,
{\em On solutions of analytic equations,} Invent. Math.  {\bf 5} (1968), 277--291.

\bibitem[BF08]{BF08}
H. U. Boden and S. Friedl,
{\em Metabelian $\SL(n,\CC)$ representations of knot groups,} Pacific J. Math.  {\bf 238} (2008), 7--25.

\bibitem[BF11]{BF11}
H. U. Boden and S. Friedl,
{\em Metabelian $\SL(n,\CC)$ representations of knot groups II, fixed points,} Pacific J. Math.  {\bf 249} (2011), 1--10.

\bibitem[BHK05]{BHK05}
H. U. Boden, C. M. Herald, and P. Kirk,
{\em The integer valued $SU(3)$ Casson invariant for Brieskorn spheres,} J. Differ. Geom.  {\bf 71} (2005), 23--83.

\bibitem[BB07]{BB07}
S. Boyer and M. Boileau,
{\em On character varieties, sets of discrete characters, and non--zero degree maps}, to appear in Amer. J. Math., arXiv {math.GT/0701384}.

\bibitem[Bu90]{Bu90}
G. Burde,
{\em $SU(2)$-representation spaces for two-bridge knot groups,} Math. Ann. {\bf 288} (1990), no. 1, 103--119

\bibitem[BZ85]{BZ85}
G. Burde and H. Zieschang,
{Knots,} {\em de Gruyter Studies in Mathematics,} {\bf 5}, Walter de Gruyter \& Co., Berlin, 1985.

\bibitem[CS83]{CS83}
M. Culler and P. B. Shalen,
{\em Varieties of group representations and splittings of 3-manifolds,} Annals of Math. {\bf 117} (1983), 109--146.

\bibitem[FK06]{FK06}
 S. Friedl and T. Kim,
{\em The Thurston norm, fibered manifolds and twisted Alexander polynomials}, Topology {\bf 45} (2006), 929--953.

\bibitem[FV10]{FV10}
S. Friedl and  S. Vidussi,
 {\em A survey of twisted Alexander polynomials}, in The Mathematics of Knots: Theory and Application edited by M. Banagl and D. Vogel (2010)


\bibitem[Gol84]{Gol84}
W. Goldman,
{\em The symplectic nature of the fundamental groups of surfaces}, Adv. in Math. {\bf 54} (1984), 200--302.

\bibitem[GAS91]{GAS91}
F. Gonz\'alez-Acu\~na and H. Short, {\em Cyclic branched coverings of knots and homology
spheres}, Revista Math. 4 (1991), 97--120

\bibitem[Gor72]{Gor72}
C. McA. Gordon, {\em Knots whose branched cyclic coverings have periodic homology}, Trans. Amer. Math. Soc. 168 (1972), 357--370.

\bibitem[HK97]{HK97}
M. Heusener and E. Klassen,
{\em Deformations of dihedral representations,} Proc. Amer. Math. Soc. {\bf 125} (1997), no. 10, 3039--3047

\bibitem[HP06]{HP06}
M. Heusener and J. Porti,
{\em  The variety of characters in $PSL2(\CC),$} Bol. Soc. Mat. Mexicana (3) {\bf 10} (2004), Special Issue, 221--237.

\bibitem[HP05]{HP05}
M. Heusener and J. Porti,
{\em Deformations of reducible representations of 3-manifold groups into $PSL(2,\CC)$,} Alg. Geom. Top. {\bf 5} (2005), 965--997.

\bibitem[HPS01]{HPS01}
M. Heusener, J. Porti, and E. Su\'arez Peir\'o,
{\em Deformations of reducible representations of 3-manifold groups into $SL(2,\CC)$,} J. Reine Angew. Math. {\bf 530} (2001), 191--227.

\bibitem[HS97]{HS97}
P. J. Hilton and U. Stammbach,
{\em A Course in Homological Algebra}, second edition, Springer Graduate Texts in Mathematics (1997).

\bibitem[LM85]{LM85}
A. Lubotzky and A. R. Magid,
{\em Varieties of representations of finitely generated groups}, Memoirs of the Amer. Math. Soc. Vol.  58, No. 336 (1985).

\bibitem[Lu73]{Lu73}
D. Luna,
{\em Slices \'etale},  Bull. Soc. Math. Fr., Supp. M\`em. {\bf 33} (1973) 81--105.

\bibitem[Ri90]{Ri90}
R. Riley, {\em Growth of order of homology of cyclic branched covers of knots}, Bull.
London Math. Soc. 22 (1990), 287--297.

\bibitem[Ro76]{Ro76}
D. Rolfsen,
Knots and Links, {\em Mathematics Lecture Series} {\bf 7},
Publish or Perish, Berkeley, CA (1976).

\bibitem[Sh95]{Sh95}
I. R. Shafarevich,
{\em Basic Algebraic Geometry I},  Springer (1995).

\bibitem[Si09]{Si09}
A. Sikora,
{\em Character varieties}, to appear in Trans. Amer. Math. Soc., arXiv {math.GT/0902.2589},


\bibitem[SW02]{SW02}
D. Silver and S. Williams, {\em Mahler measure, links and homology growth}, Topology 41 (2002), no. 5, 979--991.

\bibitem[We64]{We64}
A. Weil,
{\em Remarks on the cohomology of groups,} Annals of Math. (2) {\bf 80} (1964) 149--157.

\bibitem[Ya11]{Ya11}
Y. Yamaguchi, {\em On the twisted Alexander polynomial for metabelian $\sl(2,\CC)$–representations
with the adjoint action}, RIMS Kokyuroku 1747 (2011), 157--184.

\bibitem[Wa94]{Wa94}
M. Wada, {\em Twisted Alexander polynomial for finitely presentable groups}, Topology 33, no. 2:
241--256 (1994)
\end{thebibliography}
\end{document}